\documentclass{article} % replace ECP by EJP if needed.  
\usepackage{pict2e,amssymb,amsthm,amsmath}

\begin{document}

\title{The Compulsive Gambler Process}

%\DEDICATORY{Dedicated to the memory of ...} % Optional

%%%%%%%%%%%%%%%%%%%%%%%%%%%%%%%%%%%%%%%%%%%%%%%%%%%%%%%%%%%%%%%%%%%
%%                                                               %%
%% Authors (please edit and customize):                          %%
%%                                                               %%
%%%%%%%%%%%%%%%%%%%%%%%%%%%%%%%%%%%%%%%%%%%%%%%%%%%%%%%%%%%%%%%%%%%

\author{%
  David~Aldous\\U.C. Berkeley, USA.\\
   aldous@stat.berkeley.edu 
  \and 
  Daniel~Lanoue\\U.C. Berkeley, USA.\\ dlanoue@math.berkeley.edu
\and
Justin~Salez\\Universit\'{e} Paris 7, France. \\
salez@math.univ-paris-diderot.fr
}

\maketitle

%\KEYWORDS{coalescent; exchangeable; interacting particle system; martingale; social dynamics}
 % Separate items with ;

%\AMSSUBJ{60J27; 60K35} % Edit. Separate items with ;
%\AMSSUBJSECONDARY{FIXME:} % Optional, separate items with ;

\begin{abstract}
In the {\em compulsive gambler} process there is a finite set of agents who meet pairwise at random times 
($i$ and $j$ meet at times of a rate-$\nu_{ij}$ Poisson process) and, upon meeting, play an instantaneous fair game in which one wins the other's money.  We introduce this process and describe some of its basic properties.
Some properties are rather obvious (martingale structure;  comparison with Kingman coalescent) while others are more subtle (an ``exchangeable 
over the money elements" property, and a construction reminiscent of the Donnelly-Kurtz look-down construction).  Several directions for possible future 
research are described. One -- where agents meet neighbors in a  sparse graph -- is studied here, and 
another -- a continuous-space extension called the {\em metric coalescent} -- is studied in Lanoue (2014).
\end{abstract}

 \renewcommand{\Pr}{{\mathbb{P}}}
\newcommand{\Agents}{\mathbf{Agents}}
\newcommand{\bX}{\mathbf{X}}
\newcommand{\bS}{\mathbf{S}}
\newcommand{\TT}{{\mathcal T}}
\newcommand{\FF}{{\mathcal F}}
\newcommand{\GG}{{\mathcal G}}
\newcommand{\Reals}{\mathbb R}
\newcommand{\Ints}{\mathbb Z}
\newcommand{\Tree}{\mathbb T}
\newcommand{\bx}{\mathbf{x}}
 \newcommand{\sfrac}[2]{{\textstyle\frac{#1}{#2}}}
 \newcommand{\Ex}{{\mathbb E}}
\newcommand{\var}{\mathrm{var}}
\newcommand{\PP}{{\mathcal P}}
\newcommand{\Pfs}{\PP_{\mathrm{fs}}(S)}
\newcommand{\supp}{\mathrm{support}}
\newcommand{\ER}{Erd\H{o}s-R\'enyi }

\newtheorem{Theorem}{Theorem}[section]
\newtheorem{corollary}[Theorem]{Corollary}
\newtheorem{proposition}[Theorem]{Proposition}

\newtheorem{lemma}[Theorem]{Lemma}

\section{Introduction}
The style of models known to probabilists as 
Interacting Particle Systems (IPS) \cite{MR776231}
have found use in many fields across the mathematical and social sciences.
Often the underlying
conceptual picture is of a social network, where individual ``agents" 
meet pairwise and update their ``state" (opinion, activity etc) in a way 
depending on their previous states.  
This picture motivates a precise general setup we call  
{\em Finite Markov Information Exchange (FMIE) processes} 
\cite{MR3102546}.
Consider a set $\Agents$ of $n$   agents
and a nonnegative array 
$(\nu_{ij})$, indexed by unordered pairs $\{i,j\}$, 
which is irreducible 
(i.e. the graph of edges corresponding to strictly positive entries is connected).  
Assume
\begin{itemize}
\item Each unordered pair $i,j$ of agents with $\nu_{ij} > 0$  
meets at  the times of a rate-$\nu_{ij}$  Poisson process, 
independent for different pairs.
\end{itemize}
Call this collection of Poisson processes the {\em meeting process}; 
the array $(\nu_{ij})$ specifies the {\em meeting model}.
A specific FMIE is a specific rule (deterministic or random) for updating states, so this encompasses most of the familiar IPS models such as the voter model and contact process.  
But our emphasis differs from the classical emphasis of IPS in several ways; the states are typically numerical rather than categorical, the number $n$ of agents is finite 
(though we consider $n \to \infty$ asymptotics) 
and we focus on obtaining rough results for 
general meeting rates 
rather than sharp results for very specific meeting rates.

One specific FMIE model is the
{\em averaging process}  \cite{MR2908618} in which  
 agents initially have different amounts of money; whenever two agents meet, they share their combined money equally.  
In this paper we introduce and study a conceptually opposite model, 
the {\em compulsive gambler process}.  In the ``simple" form of the model, 
agents each start with one unit money. 
When two agents meet, if they each have non-zero money (say amounts $a$ and $b$) 
then they instantly play a fair game in which one agent acquires the combined amount $a+b$ (so with probabilities $a/(a+b)$ and $b/(a+b)$) 
respectively).  In the ``standardized" form of the model the initial fortunes can be non-uniform, and we scale so that the total money equals $1$.

This is an invented model for which we do not claim realism\footnote{Any perceived analogy between 
averaging/compulsive gambler models
 and socialism/capitalism is entirely the reader's responsibility.}, 
 but we do claim some mathematical interest as an intermediary between 
IPS theory and coalescent theory.

\subsection{Elementary observations}
\label{sec:outline}
First consider a fixed meeting model on $n$ agents.
Write $\bX(t) = (X_i(t), i \in \Agents)$ for the time-$t$ configuration of the
simple compulsive gambler process; agent $i$ has $X_i(t)$ units of money.
The following assertions are true, and mostly obvious.
We will give proofs, where necessary,  and some crude quantifications in section \ref{sec:basic} 
(Lemmas \ref{L1} and \ref{L2}).

\noindent
(i) $\bX(t)$ is a finite-state continuous-time Markov chain which,
at some a.s. finite random time $T$, reaches 
 some absorbing configuration $\bX^*$ in which there is some random non-empty set $\TT$ of agents who are {\em solvent}, i.e. have non-zero money.

\noindent
(ii) If $\nu_{ij} > 0$ for all $j \ne i$ then $|\TT| = 1$ a.s., and we call $T$ 
the {\em fixation time}.
Furthermore, because each $(X_i(t), 0 \le t < \infty)$ is a martingale we have 
$\Pr(\TT = \{i\} )= \Pr(X^*_i = n)  = 1/n$ for each agent $i$.

\noindent
(iii)  If $\nu_{ij} = 0$ for some $j \ne i$ then $\Pr ( |\TT| = 1)$ is strictly between $0$ and $1$.

\medskip
\noindent
These facts suggest more quantitative questions to ask,
in the setting of a sequence $(\nu^{(n)}_{ij})$ of meeting models with 
$n \to \infty$.  
How does $T^{(n)}$ behave? 
In case (iii), how do $|\TT^{(n)}|$ and the distribution 
$\Pr(X^*_i \in \cdot \vert i \in \TT^{(n)})$ of a typical ``final fortune"  behave?   
In either case we can ask how the process of the number of solvent agents
\[
N(t) : = | \{i: X_i(t) > 0\}| 
\]
behaves over $0 \le t \le T$.  
If the meeting model has some spatial structure then what can we say about the spatial structure of the set of  solvent agents at time $t$?
We continue this discussion of research directions in section \ref{sec:directions}.  

\subsection{Techniques}
It turns out that a surprising 
variety of techniques can be exploited in the study of the compulsive gambler process.
Amongst these techniques, to be described in section \ref{sec:basic}, the most natural are martingale results (Lemma \ref{L:mg}) and elementary bounds obtained by comparison with the Kingman coalescent (e.g. Lemma \ref{L1}).  
Less obvious is Lemma \ref{L:augment}: instead of making the random choices of game-winners at the meeting times, we can insert initial randomness and then have a deterministic rule for game-winners.
And in the ``simple" case that construction has a  symmetry property (Lemma \ref{L:set-augment}):
 the deterministic rule is based on a uniform random labeling of initial currency notes as $1,\ldots,n$, and conditional on the configuration of fortunes at time $t$, the allocation of note-labels to agents is uniform random.
This last method seems very similar to methods used in the study of 
{\em exchangeable coalescents} \cite{MR2574323,MR2253162,EC},
though the precise relation is not clear to us and we have not used results from that theory.

\subsection{Directions for future research}
\label{sec:directions}
Our main purpose  is to lay the groundwork for future research 
by describing explicitly these techniques (section \ref{sec:basic}).
In this paper we pursue analysis in only one direction, by studying the setting
where the meeting model is that agents meet neighbors in a sparse graph (section \ref{sec:sparse}).
Here are some other directions of current or future research.

\paragraph{The metric coalescent.}  This concerns a continuous-space extension. 
Take a suitable space $S$, write $\PP(S)$ for the space of probability measures $\mu$ on $S$ and write $\Pfs \subset \PP(S)$ for the subspace of finite support probability measures.  
Consider a symmetric function $\nu: S^2 \to \Reals_{\ge 0}$.
For any $\mu \in \Pfs$, we can consider the standardized compulsive gambler 
process for which the set $\Agents$ is the support $\{s_1,\ldots, s_n\}$ of $\mu$, the meeting rates are the $\nu(s_i,s_j)$, 
and the initial distribution of money is $\mu$; and moreover we can 
regard the states of the process as elements of $\Pfs$.  
So we can reconsider the standardized compulsive gambler 
process as a Markov process, specified by the function $\nu$, whose state space is (all of) $\Pfs$. 
Then we can ask, inspired by the Kingman coalescent and its extensions \cite{MR2253162}, whether it makes sense to imagine this process starting with a general (in particular, non-atomic) initial 
state $\mu_0 \in \PP(S)$.  
This topic is studied in detail in \cite{lanoue-MC} in the context of
 a complete separable locally finite metric space $(S,d)$
and meeting rates of the form 
\[ \nu(s,s^\prime) = \phi( d(s,s^\prime)) \]
for some continuous  function $\phi(\cdot) > 0$.  
The main result of \cite{lanoue-MC} is that, under the condition 
$\lim_{x \downarrow 0} \phi(x) = \infty$,  the standardized compulsive gambler process on $\Pfs$ extends to a Feller process 
(the {\em metric coalescent}) on all of $\PP(S)$.  
In particular, for an initial $\mu_0 \in \PP(S)$ with compact support, 
the metric coalescent process 
$(\mu_t, 0 \le t < \infty)$ has finite support at each $t_0 > 0$, evolves as the compulsive gambler process over $[t_0, \infty)$ and satisfies the initial condition 
\[ \Pr ( \lim_{t \downarrow 0} \mu_t = \mu_0 \mbox{ in } \PP(S) ) = 1 . \]
A key ingredient in the proof is Corollary \ref{C:set-aug} below.  
In informal language, Corollary \ref{C:set-aug}  says that for the simple compulsive gambler process, instead of determining the game winners at the meeting times, we can do so via initial randomization, as follows.  
Initially each agent has a currency note with a random serial number; when two solvent agents meet, each has a collection of notes, but now the winner is always the owner of the lowest-ranked note, ranking by serial number.
In other words we start with a uniform random ordering $s_1,\ldots,s_n$  of $\Agents$, ranked by serial number of note.
In the continuous-space setting we can do the same construction but starting with i.i.d. ($\mu_0$) random samples $s_1,\ldots,s_n$.  
For each $n$ we now have a $\Pfs$-valued process 
$(\mu^{(n)}_t, 0 \le t < \infty)$. 
But as $n$ varies these processes have a natural coupling and the 
metric coalescent can be constructed as the a.s. $n \to \infty$ limit process.

\paragraph{Infinite discrete space.}  For a countable infinite set $\Agents$, the simple compulsive gambler process 
is well-defined under certain conditions, for instance if 
\begin{equation}
\nu^* := \sup_i \sum_{j \neq i} \nu_{ij} < \infty . 
\label{Zd1}
\end{equation}
In particular, for an infinite vertex-transitive bounded degree graph, 
with meeting rates $\nu_e \equiv 1$ for edges $e$, there is a random set 
$\TT$ of agents who remain solvent in the  $t \to \infty$ limit, and one can seek to calculate the density $\Pr(i \in \TT)$ of that set -- see section \ref{sec:sparse} for the case of the $r$-regular tree.
For another direction, 
consider the case where $\Agents = \Ints^d$ and the meeting rates are 
\begin{equation}
 \nu_{ij} = ||j - i||^{-\alpha} 
\label{alpha}
\end{equation}
for some $\alpha > d$, implying (\ref{Zd1}). 
Consider the mean density of solvent agents at time $t$
\[ \rho(t) := \Pr(X_i(t) \neq 0) \]
and the conditional distribution $X^*(t)$ of $X_i(t)$ given $X_i(t) \neq 0$, 
for which $\Ex X^*(t) = 1/\rho(t)$ because $\Ex X_i(t) \equiv 1$.  
Heuristic arguments,  based on supposing the positions of solvent agents do not become ``clustered", suggest that
\[ \rho(t) \asymp t^{-\beta} \mbox{ for } 
\beta = \sfrac{d}{\alpha} .\] 
It is then plausible that 
\[ \rho(t) X^*(t) \to_d Z, \ \mbox{ for some $Z$ such that } \Ex Z = 1 \] 
and then that the process has a scaling limit, the limit being a  process whose states are 
(locally finite support) measures on $\Reals^d$.

\paragraph{Other finite-agent meeting models.}  For the complete graph meeting model 
($\nu_{ij} \equiv 1$) the compulsive gambler process is essentially just the Kingman coalescent.  
 On the $d$-dimensional discrete torus $\Ints^d_m$ 
one could reconsider meeting rates as at (\ref{alpha}). 
By the heuristics above, for $\alpha > d$ we expect the fixation time $T$ to scale as 
$\rho^{-1}(m^{-d}) =  m^\alpha$. 
In this setting it makes sense to consider also the case $\alpha < d$, but in this case an agent will tend to meet distant agents rather than nearby ones, and by comparison with the Kingman 
coalescent we expect $T$ to scale as the total meeting rate 
$\sum_{j \ne i} \nu_{ij}$ of a given agent, that is as 
$m^{d - \alpha}$.  
Finally, by the techniques of  section \ref{sec:sparse} one can calculate 
the asymptotic density (\ref{LWC2}) of solvent agents under the sparse 
\ER meeting model, a result which can alternatively be seen in terms of the {\em short-time} behavior of the Kingman coalescent 
(section \ref{sec:ERKC}).

\section{Four basic techniques}
\label{sec:basic}
We now abbreviate ``compulsive gambler" to CG.
Fix a meeting model $(\nu_{ij})$ on a set $\Agents$ of $n$ agents.
As in section \ref{sec:outline}, write $\bX(t) = (X_i(t), i \in \Agents)$ for the time-$t$ configuration of the
CG process; agent $i$ has $X_i(t)$ units of money. 
And write
\[
N(t) : = \vert \{i: X_i(t) > 0\} \vert
\]
for the number of {\em solvent} agents.
The CG process is specified by its transition rates.
For each ordered distinct pair $(j,k)$ with $\min(x_j,x_k) > 0$,
\begin{equation}
 \bx \to \bx^{(j,k)} \mbox{ at rate } \nu_{jk} \sfrac{x_j}{x_j + x_k}; \quad \mbox{ where } 
\label{def:rates}
\end{equation}
\[ \bx^{(j,k)}_i = x_i, i \ne j,k; \quad \bx^{(j,k)}_j = x_j+x_k; \quad 
\bx^{(j,k)}_k = 0 . \] 
In full generality the state space conists of all configurations 
$\bx = (x_1,\ldots, x_n)$ with $x_i \ge 0 \ \forall i$.  
For the {\em simple} CG process the initial state is 
$X_i(0) = 1 \ \forall i \in \Agents$.
For the {\em standardized} CG process there is an initial state 
$\bx = (x_i)$ with $x_i \ge 0 \ \forall i$ and $\sum_i x_i = 1$.  
Clearly $\sum_i X_i(t) \equiv n$ in the simple case and 
$\sum_i X_i(t) \equiv 1$ in the standardized case.  
Results in this section \ref{sec:basic} hold 
hold in both simple and standardized 
cases unless otherwise stated.

\subsection{Martingale properties}
We first record some notation for the elementary stochastic calculus of 
integrable bounded variation processes.  Such a process $(Z_t)$ has a Doob-Meyer
decomposition $Z_t = M_t + A_t$, where $(M_t)$ is a martingale and $(A_t)$ is predictable, which can be written in differential notation as 
$dZ_t = dM_t + dA_t$.  
To avoid introducing new symbols, we write 
$\Ex (dZ_t \vert \FF_t)$ for $dA_t$.

\begin{lemma}
\label{L:mg} 
For any meeting process: \\
(i) $(X_i(t), 0 \le t < \infty)$ is a martingale. \\
(ii) For $j \ne i$, $(X_i(t)X_j(t), 0 \le t < \infty)$ is a supermartingale. \\
(iii) For $ f:\Agents \to \Reals$ write $M_f(t) = \sum_i f(i) X_i(t)$.  
Given a metric $d$ on $\Agents$ write 
\[ L_f := \max_{j \ne i} \sfrac{|f(j) - f(i)|}{d(i,j)}, \quad 
\nu^* := \max_{j \ne i} \nu_{ij} d^2(i,j) . \]
Then $(M_f(t), 0 \le t < \infty)$ is a martingale, and for a
 standardized CG process, 
\[\Ex M^2_f(t) - M^2_f(0) \le \sfrac{1}{2} \nu^* L_f^2 t .\]
\end{lemma}
\begin{proof}
(i) and (ii) are straightforward.  For (iii), $M_f(\cdot)$ is a martingale, and we calculate
\[
 \Ex(dM^2_f(t)|\FF(t)) = dt \ 
 \sum_{\{i,j\}} \nu_{ij} X_i(t)X_j(t) (f(j)-f(i))^2 ,
\]
the sum being over unordered pairs.
From (ii) we have $\Ex X_i(t)X_j(t) \le x_ix_j$, where $\bx = (x_i)$ 
is the initial configuaration, so taking expectation
\begin{eqnarray*}
 \Ex (dM^2_f(t)) &\le& dt \ \times \sfrac{1}{2} 
\sum_i \sum_{j \ne i} ( \nu^* d^{-2}(i,j) \times (L_f d(i,j))^2 \times x_ix_j )
\\
&=& dt \ \times \sfrac{1}{2} \nu^* L_f^2 \sum_i \sum_{j \ne i} x_ix_j \\
&\le& dt \ \times \sfrac{1}{2} \nu^* L_f^2.
\end{eqnarray*} 
\end{proof}

An explicit formula for $\Ex M^2_f(t)$ is given in Lemma \ref{L:m12}.

\subsection{The Kingman coalescent}
In the particular case
\[ \nu_{ij} = 1, \quad j \ne i \] 
of the meeting model, the compulsive gambler process
is essentially the well-studied 
{\em Kingman coalescent} \cite{MR2574323}. 
In this case the process $(N(t), 0 \le t < \infty)$ is the ``pure death" 
Markov chain, started at $n$,  with transition rates 
$q_{m,m-1} = {m \choose 2}$, from which it immediately follows 
that the  fixation time 
\begin{equation}
T : = \min \{t: N(t) = 1\}
\label{def:fixation}
\end{equation}
is a.s. finite with expectation
\begin{equation}
\Ex T = \sum_{m=2}^n 1/{m \choose 2} = 2(1 - n^{-1}) .
\label{Kingman:eq}
\end{equation}
Here is a simple application.
\begin{lemma}
\label{L1}
Consider a meeting model for which
$\delta : = \min_{j\ne i} \nu_{ij} > 0$. \\
(i) The fixation time $T$ satisfies $\Ex T \le 2/\delta$; \\
(ii) $\Pr(N(t) > r) \le \sfrac{2}{r\delta t}, \ r \ge 2$;\\
(iii) $\Ex N(t) \le \sfrac{C}{t \delta}, \ 0 < t < \infty$ 
for some numerical constant $C < \infty$;\\
(iv)  Let $L$ be the agent who has acquired all the money at time $T$.
In the simple CG process, $\Pr(L=i) = 1/n, \quad \forall i \in \Agents$.  
In the standardized CG process with initial configuration $(x_i)$, 
$\Pr(L=i) = x_i, \quad \forall i \in \Agents$.  
\end{lemma}
\begin{proof}
Although the process $(N(t))$ is typically not Markov, when 
$N(t) = m$ the conditional intensity of a transition $m \to m-1$ is at least
$\delta {m \choose 2}$, so (i) follows by comparison with 
the Kingman chain result (\ref{Kingman:eq}).
Similarly, write $T_{(r)} = \min \{t:  N(t) \le r\}$ and 
$T^{{\mbox{{\tiny King}}}}_{(r)}$ for the corresponding quantity for the 
Kingman chain. Then
\begin{eqnarray*}
\Pr (N(t) > r) &=& \Pr(T_{(r)} > t) \\
& \le&  t^{-1} \Ex T_{(r)}\\
& \le & \delta^{-1} t^{-1} \Ex T^{{\mbox{{\tiny King}}}}_{(r)}  \mbox{ by comparison with 
the Kingman chain.}
\end{eqnarray*}
And 
\[ \Ex T^{{\mbox{{\tiny King}}}}_{(r)} = \sum_{m=r+1}^n 1/{m \choose 2} \le 2/r . \]
So (ii) follows by comparison.  A similar argument, calculating 
$\var (T^{{\mbox{{\tiny King}}}}_{(r)}) $ and using Chebyshev's inequality, 
establishes (iii).
Assertion (iv) follows from the martingale property (Lemma \ref{L:mg}(i)) 
of $(X_i(t))$, 
applying the optional sampling theorem at time $T$.
\end{proof}

\subsection{The augmented process}
\label{sec:augment}
Given a probability distribution $\pi = (\pi_i)$ on the set $\Agents$ of $n$ agents with each $\pi_i> 0$,
take independent random variables $\eta_i$ with Exponential($\pi_i$) 
distributions.  
Define a random ordering $\prec$ on $\Agents$ by
\begin{equation}
 i \prec j \mbox{ if } \eta_i < \eta_j . 
\label{defRO}
\end{equation}
This is one of several equivalent definitions of the 
{\em size-biased random ordering} \cite{MR1659532} associated with $\pi$. For instance, defining a random bijection $F: \{1,\ldots,n\} \to 
\Agents$ by 
\begin{eqnarray*}
 \Pr(F(1) = i) &=&\pi_i \\
 \Pr(F(2) = j|F(1) = i) &=& \pi_j/(1-\pi_i), \ j \ne i \\
 \Pr(F(3) = k|F(1) = i,F(2)=j) &=& \pi_k/(1-\pi_i - \pi_j), \ \{i,j,k\} \mbox{ distinct} \\
\ldots&& 
\end{eqnarray*}
and so on, 
then the size-biased random ordering could be defined as 
\begin{equation}
 i \prec j \mbox{ if } F^{-1}(i)  < F^{-1}(j) . 
\label{F-def}
\end{equation}
We want to consider the standardized CG process with some initial
configuration $\bx(0) = (x_i(0))$. 
Take the size-biased random ordering $\prec$ on  $\Agents$ associated with the probabiity distribution $\bx(0)$.
Conditional on the realization of $\prec$, 
we can define a variation of the CG process in which, when two agents $i,j$  with non-zero money meet, the winner is always the agent who comes earlier in $\prec$ 
(if $i \prec j$ then $i$ is the winner). 
In other words, the transition rates (\ref{def:rates}) become
\begin{equation}
  \bx \to \bx^{(j,k)} \mbox{ at rate } \nu_{jk} \mbox{ if } 
\min(x_j,x_k) > 0 \mbox{ and } j \prec k . 
\label{def:rates-aug}
\end{equation} 
Call this the {\em augmented process} $(\bX(t), \prec)$ 
with initial state $(\bx(0), \prec)$.  
Note that the random order $\prec$ does not change with time.
See below for a way of visualizing this process in the simple setting.
The next lemma says that unconditionally, that is when we do not see 
the realization of $\prec$, we see the CG process.
\begin{lemma}
\label{L:augment}
In the augmented process $((\bX(t), \prec), 0 \le t < \infty)$ 
with initial state $(\bx(0), \prec)$, the component $(\bX(t), 0 \le t < \infty)$ evolves as the 
standardized CG process with initial configuration $\bx(0)$. 
\end{lemma}
\begin{proof}
Recall the elementary facts that, for independent 
Exponential r.v.'s $\eta_1, \eta_2$ with rates 
$\lambda_1, \lambda_2$,
\begin{equation}
\Pr(\eta_1 < \eta_2) = \sfrac{\eta_1}{\eta_1 + \eta_2} \label{elem1}
\end{equation}
\begin{equation}
\mbox{the conditional dist. of $\eta_1$ given $\eta_1 < \eta_2$ 
is Exponential$(\eta_1 + \eta_2)$.} \label{elem2}
\end{equation}
Implement the augmented process using the order $\prec$ at (\ref{defRO}) given by 
independent Exponential($x_i(0)$) r.v.'s $(\eta_i, i \in \Agents)$. 
Write $\FF(t) = \sigma(\bX(s), 0 \le s \le t)$, and note this does not include 
the random order $\prec$.  
We claim that for each $t$
\begin{quote}
conditional on $\FF(t)$, the r.v.'s 
$(\eta_i: \ i \in \Agents, X_i(t) > 0)$ 
are independent Exponentials with rates $X_i(t)$.
\end{quote}
It is enough to check this remains true inductively over meetings. 
If agents $i$ and $j$ meet at $t$ with non-zero fortunes 
$X_i(t-)$ and $X_j(t-)$, then by the evolution rule for the augmented 
process
\begin{quote}
on the event $\{\eta_i < \eta_j\}$ we have 
$X_i(t) = X_i(t-) + X_j(t-)$ and $X_j(t) = 0$
\end{quote}
and similarly on the complementary event.  By inductive hypothesis $\eta_i$ and $\eta_j$ are independent 
Exponentials of rates $X_i(t-)$ and $X_j(t-)$; fact (\ref{elem2}) then says 
that $\eta_i$ has Exponential $X_i(t)$ distribution and the induction goes through.

Having established the claim, consider again what happens when 
agents $i$ and $j$ meet at $t$ with non-zero fortunes 
$X_i(t-)$ and $X_j(t-)$.  The probability that the update is to
$X_i(t) = X_i(t-) + X_j(t-)$ and $X_j(t) = 0$ 
(and similarly for the complementary event) is the probability of the event $\{\eta_i < \eta_j\}$, 
which by the claim and (\ref{elem1}) equals $X_i(t-)/( X_i(t-) + X_j(t-) )$.  
But this is the dynamics of the CG process.
\end{proof}

See \cite{lanoue-MC} for uses of this result in the context of the metric coalescent.

\subsection{The token process}
\label{sec:token}
For the special case of a {\em simple} CG process, there is a 
more concrete and informative  expansion  of the notion of {\em augmented process} above.
First, here is a story which might help visualize what is going on.   
(In talks we ask several audience members to each place an actual currency note on the table, so we can demonstrate the story.)
Real-world currency notes have serial numbers; imagine each agent starting with one note with a random serial number, so that the ranking (smallest to largest) of the $n$ notes is uniform random.  
When two agents with non-zero money meet, we specify that the agent who wins the game is determined as the agent who  possesses, in their collection at that time, the smallest-ranked note.  The winner adds the loser's notes 
to his pile of notes.

In the story, each agent has a set of notes, but what is relevant is not the precise serial numbers but the relative rankings of each of the $n$ serial numbers.
In the formalization below,  $(S_i(t))$ is the set of rankings of all the notes owned by agent $i$ at time $t$.

Note this story is 
consistent with the ``simple" case of Lemma \ref{L:augment}, 
which is essentially the context where we record only the relative orders of each agent's {\em smallest-ranked} note.  But in contrast to Lemma \ref{L:augment}, 
what we do next is useful only in the ``simple" context.

To formalize the story above,
given meeting rates $(\nu_{ij}, i,j \in \Agents)$ 
we first take a uniform random bijection 
$F:\{1,\ldots,n\} \to \Agents$. 
Visualize {\em tokens} $1,\ldots,n$ being randomly dealt to the agents. 
Define a process
$\bS(t) = (S_i(t), i \in \Agents)$ 
to have initial configuration
\[ S_i(0) = \{ F^{-1}(i) \}, \ i \in \Agents \]
and transition rates (copying (\ref{def:rates-aug}))
\begin{equation}
  \bS \to \bS^{(j,k)} \mbox{ at rate } \nu_{jk} \mbox{ if 
$S_j$ and $S_k$ non-empty and }
\min S_j < \min S_k 
\label{Sjk}
\end{equation}
where 
\[ S^{(j,k)}_j = S_j \cup S_k, \ S^{(j,k)}_k = \emptyset, \ S^{(j,k)}_i = S_i 
\mbox{ for } i \ne j,k . \]
So $S_i(t)$ is just the set of tokens held by agent $i$ at time $t$, 
and at a meeting the game is always won by the owner of the smallest (lowest-ranked) token.
Call $(\bS(t), 0 \le t < \infty)$ the {\em token process}, 
and write 
\[ \bX(t) = ( |S_i(t)|, \ i \in \Agents); \quad 
\FF(t) = \sigma(\bX(s), \ 0 \le s \le t) . \]
As discussed above, 
Lemma \ref{L:augment} implies 
\begin{corollary}
\label{C:set-aug}
In the token process $(\bS(t), 0 \le t < \infty)$, the process 
$\bX(t) := ( |S_i(t)|, \ i \in \Agents)$ evolves as
 the simple CG process.
\end{corollary}
As mentioned earlier, Corollary \ref{C:set-aug}  plays a key role in the
development of the metric coalescent in \cite{lanoue-MC}. 
And we will see in sections \ref{sec:ep} and \ref{sec:sparse} how it enables us to use simple intuitive arguments in our discrete setting.
Corollary \ref{C:set-aug} is reminiscent of the Donnelly-Kurtz look-down construction
\cite{MR1681126} but we do not see a precise connection.

Lemma \ref{L:set-augment} below says:  if we just see the number of tokens that each agent has, then the assignment of tokens to agents
is uniform over possible assignments.  
\begin{lemma}
\label{L:set-augment} 
In the token process, for each $t$,  the conditional distribution 
of $(S_i(t), i \in \Agents)$  given $\FF(t)$ 
is uniform over all partitions $(B_i, i \in \Agents)$ of $\{1,\ldots,n\}$  with 
$|B_i| = X_i(t) \ \forall i$.
\end{lemma}
\begin{proof}
As in the proof of Lemma \ref{L:augment},
it is enough to check that the assertion remains true inductively over meetings.
Given that $S_{j_1}(t)$ and $S_{j_2}(t)$ are non-empty, the event of a meeting of $(j_1,j_2)$ in $(t,t+dt)$ 
is independent of $(S_{j_1}(t),S_{j_2}(t))$. 
Such a meeting causes either $S_{j_1}$ or $S_{j_2}$ to become  $S_{j_1}(t) \cup S_{j_2}(t)$
and the other to become empty.  Now checking that the induction goes through reduces to checking the following 
elementary fact about merging components of uniform random partitions, which we leave to the reader.

Take $(n_i, i \in I)$ with each $n_i \ge 1$ and $\sum_i n_i = n$.
Take two elements $j_1,j_2$ of $I$, write $j_0$ for a new symbol and let 
$I^\prime := (I \setminus \{j_1,j_2\}) \cup \{j_0\}$ 
and $n_{j_0} = n_{j_1} + n_{j_2}$.  
Take a uniform random partition $(B_i, i \in I)$ of $\{1,\ldots,n\}$ into 
components with $|B_i| = n_i \ \forall i \in I$. 
Construct a random partition  $(B^\prime_i, i \in I^\prime)$ by setting 
$B_{j_0} = B_{j_1} \cup B_{j_2}$ and $B^\prime_i = B_i$ for other $i$.  
Then \\
(i) $(B^\prime_i, i \in I^\prime)$ is a uniform random partition of $\{1,\ldots,n\}$ into 
components with $|B^\prime_i| = n_i \ \forall i \in I^\prime$; \\
(ii) The event ``the minimum element of $B_{j_1}$ is smaller than the minimum element of $B_{j_2}$"
is independent of the random partition 
$(B^\prime_i, i \in I^\prime)$.

In applying this fact in our setting, the point is that the information revealed by the change in $\bX(\cdot)$ at the meeting is precisely 
the identity of $j_1,j_2$ and whether the event in (ii) occurs, but conditioning on these does not 
destroy uniformity. 
\end{proof}
A more detailed treatment is given in \cite{lanoue-MC} section 2.4.
%xxx check section number above

\subsection{Moment calculations}
Lemma  \ref{L:set-augment} allows us to do various calculations with a simple CG process, such as the second-moment calculations below.  Recall that 
(by the martingale property) we know $\Ex X_i(t) \equiv 1$.
\begin{lemma}
\label{L:m12}
For a simple CG process,
\begin{equation}
 \Ex [ X_i(t) X_j(t) ] = \exp(- \nu_{ij}t) , \ j \ne i .
\label{temp-1}
\end{equation}
\begin{equation}
  \Ex [ X_i(t)  (X_i(t) -1) ] = \sum_{j \ne i}  (1 -  \exp(- \nu_{ij}t)) . 
\label{temp-2}
\end{equation}
So for $f:\Agents \to \Reals$ we have
\begin{equation}
\mbox{\small $
\Ex \left( \sum_i f_iX_i(t) \right)^2 
= \sum_i f_i^2 \left[ 1 + \sum_{j\ne i} (1 - \exp(-\nu_{ij}t) ) \right] 
+ \sum_i \sum_{j \ne i} f_if_j \exp(-\nu_{ij}t) ) 
$}
. \label{2-moment}
\end{equation}
\end{lemma}
\begin{proof}
 Lemma \ref{L:set-augment} 
gives 
\[ \Pr( 1 \in S_i(t), 2 \in S_j(t) \vert \FF(t)) = \frac{X_i(t)X_j(t)}{n(n-1)} \]
and taking expectation 
\[ \Pr( 1 \in S_i(t), 2 \in S_j(t)) =   \frac{1}{n(n-1)}      \Ex [ X_i(t)  X_j(t) ]  .
\]
From the dynamics of the token process, the event 
$\{ 1 \in S_i(t), 2 \in S_j(t) \}$ happens if and only if 
$F(1) = i, F(2) = j$ and $\tau_{ij} > t$, 
where $\tau_{ij}$ is the first meeting time of $i$ and $j$.  So 
\[ \Pr( 1 \in S_i(t), 2 \in S_j(t)) = \frac{1}{n(n-1)}  \ \Pr(\tau_{ij} > t) .\] 
These last two identities give (\ref{temp-1}).  
One can deduce (\ref{temp-2}) from (\ref{temp-1}) and the ``martingale" 
fact $\Ex |S_i(t)| = 1$, but let us see how it follows by the same kind of argument as above.   Lemma \ref{L:set-augment} 
gives 
\[ \Pr( \{1,2\} \subseteq  S_i(t) \vert \FF(t)) = \frac{X_i(t)(X_i(t)-1)}{n(n-1)} \]
and taking expectation 
\[ \Pr( \{1,2\} \subseteq  S_i(t) ) =  \frac{1}{n(n-1)}      \Ex [ X_i(t)  (X_i(t) -1) ] .
\]
From the dynamics of the token process, the event 
$\{ \{1,2\} \subseteq  S_i(t) \}$ happens if and only if 
$F(1) = i$ and $ F(2) = $ some $j$ for which $\tau_{ij} \le t$, and so 
\[ \Pr( \{1,2\} \subseteq  S_i(t) ) = \frac{1}{n(n-1)}  \ 
\sum_{j \ne i} \Pr(\tau_{ij} \le t) . \]
These last two identities give (\ref{temp-2}).  
Finally, (\ref{2-moment}) follows from  (\ref{temp-1}) and  (\ref{temp-2}) by expanding the square.
\end{proof}

\subsection{Elementary properties of the simple CG process}
\label{sec:ep}
Here we prove the  remaining ``mostly obvious" assertions about the 
simple CG process
from section \ref{sec:outline}, and some minor extensions.
Fix a meeting model $(\nu_{ij})$ on $n$ agents.
Write $G$ for the graph whose edges are the pairs $(i,j)$ with $\nu_{ij}>0$.   
Recall $G$ is connected by assumption.
An {\em anticlique} (or {\em independent set}) in $G$ is a set $A$ of vertices
such that there is no edge with both end-vertices in $A$.  
There is a finite set of configurations $\bx$ that can be reached by the 
simple CG process.  Such a configuration is absorbing if and only if 
$\{i: x_i \ge 1\}$ is an anticlique. 
The process must reach some absorbing configuration at some a.s. finite time $T$,
because  $N(t)$  (the number of solvent agents) decreases by one every time the configuration changes.
Write $\TT$ for the random set of agents with non-zero money at $T$.
\begin{lemma}
\label{L2}
For the simple CG process:\\
(i) $\Ex T \le (n-1)/\delta$, where $\delta:= \min\{\nu_{ij}: \ \nu_{ij} > 0 \}$.\\
(ii) For each pair $\{i,j\}$ with $\nu_{ij} = 0$ we have 
$\Pr ( i \mbox{ and } j \in \TT) \ge \frac{2}{n(n-1)}$. \\
(iii) $\Pr (i \in \TT) \ge \frac{1}{1+d(i)}$ and so $\Ex |\TT| \ge \sum_i \sfrac{1}{1+d(i)}$, 
where  $d(i)$ is the degree of vertex $i$ in $G$.\\
(iv) $\Pr (|\TT| = 1) > 0$.
\end{lemma}
\noindent
Note that, if $\nu_{ij} = 0$ for some pair $\{i,j\}$, 
then (ii)  implies $\Pr (|\TT| = 1) < 1$.

\begin{proof}
If $N(t) = m$ and the configuration is not an anticlique then 
the conditional intensity of a transition $m \to m-1$ is at least
$\delta$, implying (i) by comparison with the pure death process with 
constant transition rate $\delta$.
For (ii) consider the token process from section \ref{sec:token}.
If $\nu_{ij} = 0$ then with probability $1/{n \choose 2}$  agents
$i$ and $j$ have tokens $1$ and $2$; if so,
then neither can lose a game, so both must end in $\TT$.
Similarly for (iii), with probability $1/(1+d(i))$ agent $i$'s token is smaller  than all its $d(i)$ neighbors' tokens; if so, then agent $i$ cannot lose a game, 
implying $i \in \TT$.
So $\Pr(i \in \TT) \ge 1/(1+d(i))$, which is (iii).
For (iv), consider a spanning tree for $G$.
We can order its edges as 
\[ e_1 = (\ell_1,v_1), \ e_2 = (\ell_2,v_2), \ \ldots \ 
   e_{n-1} = (\ell_{n-1},v_{n-1}) \]
in such a way that each $\ell_i$ is a leaf of the subtree in which 
edges $e_1,\ldots,e_{i-1}$ have been deleted.  
With non-zero probability, the first $n-1$ meetings in the meeting process
are over the edges $e_1, \ldots, e_{n-1}$ in that order; 
and with non-zero probability, the game involving $(\ell_i, v_i)$ is won by 
$v_i$ for each $i$.  
If this happens then $v_{n-1}$ ends up with all the money.
\end{proof}

\section{The sparse graph setting}
\label{sec:sparse}
Consider a connected finite graph $G$ with $n$ vertices and which is $r$-regular, for $r \ge 3$ (so if $r$ is odd then $n$ must be even).  Take the set $\Agents$ as the vertices of $G$, and the meeting rates as 
\begin{eqnarray}
 \nu_{ij} &=& 1 \mbox{ if $(i,j)$ is an edge} \label{nu-constant} \\
&=& 0 \mbox{ if not}. \nonumber
\end{eqnarray}
As observed in section \ref{sec:outline}, the simple CG process must terminate in a random configuration 
$\bX^*$ with some random set $\TT$ of solvent agents.  
We study the density of solvent agents:
\[ \rho(G) : = n^{-1} \Ex |\TT| . \]
What are the possible values of $\rho(G)$, in terms of $n$ and $r$?
Consider first the lower bound.
\begin{lemma}
\label{LsparseLB}
(i) $\rho(G) \ge \frac{1}{r+1}$. \\
(ii) If $n$ is a multiple of $r$ then there exists a graph $G$ such that
\[ \rho(G) \le \sfrac{1}{r} (1 + \sfrac{2\kappa_r}{r-1})  \]
where $\kappa_r$, defined by (\ref{kappa-r}) below, is such that $\kappa_r \uparrow \kappa_\infty < \infty$ as $r \uparrow  \infty$.
\end{lemma}
\begin{proof}
Assertion (i) repeats Lemma \ref{L2}(iii). 
For (ii), consider the graph $G$ constructed as follows. 
Take $n/r$ disjoint graphs $C_1, \ldots, C_{n/r}$, each being the complete graph on $r$ vertices with one edge $(a_i,b_i)$ removed. 
Then add edges $(b_1,a_2), (b_2,a_3), \ldots, (b_{n/r}, a_1)$ to make $G$.

For each $1\le i \le n/r$, the only possible way for $\TT$ to contain more than one vertex of $C_i$
is for $\TT$ to contain the two vertices $\{a_i,b_i\}$ (because any other pair of agents in $C_i$ will meet).  It follows that 
\[ \rho(G) \le \sfrac{1}{r} (1 + q) \]
where $q$ is defined as the probability of the event 
$\{a_i,b_i\} \subseteq \TT$.
To study $q$ 
we use the token process representation from section \ref{sec:token}.
The only way that the event 
$\{a_i,b_i\} \subseteq \TT$ can occur is if one of $\{a_i,b_i\}$ has the smallest token amongst agents $C_i$.  So 
\[ q \le \sfrac{2}{r-1} \sum_{m=2}^r q_m \]
where $q_m$ is the probability that $\TT$ does contain both 
$a_i$ and $b_i$, given that $a_i$ has the smallest and $b_i$ has  the $m$'th smallest token amongst agents $C_i$.  
But this latter event can only happen if, 
in the sequence of games involving the agents initially holding these $m$ tokens, $b_i$ is never involved, 
which has probability 
 \[ \prod_{i=0}^{m-3}   \frac{ {m-i-1 \choose 2} }{ {m - i \choose 2} - 1} = 
\sigma_m ,\mbox{ say} . \]
So now we have shown 
\[ q \le \sfrac{2}{r-1} \sum_{m=2}^r \sigma_m  .
\]
But $\sigma_m$ decreases as order $m^{-2}$, establishing the bound in (ii) 
for 
\begin{equation}
\kappa_r =  \sum_{m=2}^r \sigma_m  .
\label{kappa-r}
\end{equation}
\end{proof}

Finding somewhat tight upper bounds complementary to those in Lemma \ref{LsparseLB} 
seems more difficult.  As noted in section \ref{sec:directions}, we can consider the simple CG process on the infinite $r$-ary tree $\Tree_r$, 
and the random set $\TT$ of solvent agents in the $t \to \infty$ limit 
has some density
\[ \rho(\Tree_r) := \Pr(i \in \TT) . \]
It is well-known \cite{MR2650042}  that there exist, for fixed $r \ge 3$, sequences 
$(G_{n,r}, n \ge n_0(r))$ of $r$-regular $n$-vertex connected graphs 
(derived e.g. from typical realizations of random $r$-regular graphs)
which converge in the sense of local weak convergence (Benjamini-Schramm 
convergence) to $\Tree_r$, and for such a sequence we will have
\begin{equation}
 \lim_n \rho(G_{n,r}) = \rho(\Tree_r) . 
\label{LWC1}
\end{equation}
By analyzing the CG process on $\Tree_r$ we will show
(Corollary \ref{C:r-tree}) that $\rho(\Tree_r) \sim 2/r$ as $r \to \infty$.
Granted that result, we can summarize Lemma \ref{LsparseLB} and the discussion above as follows.
\begin{proposition}
Define 
\begin{eqnarray*}
a^*(r) &=& \sup_{(G_{n,r})} \limsup_n \rho(G_{n,r}) \\
a_*(r) &=& \inf_{(G_{n,r})} \liminf_n \rho(G_{n,r}) 
\end{eqnarray*} 
the {\em sup} and {\em inf} over  sequences 
$(G_{n,r}, n \ge n_0(r))$ of $r$-regular $n$-vertex connected graphs. 
Then
\begin{eqnarray*}
a_*(r) &\sim & \frac{1}{r} \mbox{ as } r \to \infty\\
a^*(r) & \ge & \frac{2 - o(1)}{r} \mbox{ as } r \to \infty .
\end{eqnarray*} 
\end{proposition}
We conjecture that in fact
$a^*(r) \sim 2/r$ as $r \to \infty$, in other words that locally tree-like 
graphs are asymptotically extremal for this problem.

\subsection{Finite trees}
Consider the simple CG process on a finite tree $\Tree$, 
with the constant meeting rates (\ref{nu-constant}) over edges. 
We establish a recursion, Lemma \ref{L:rec}, for the distribution of 
$X_{(\Tree,o)}(t)$, the fortune of agent $o$ at time $t$.  
The CG process uses only the first meeting times $\tau_e$ across edges, which are independent with Exponential(1) distribution; by a deterministic time-change we can suppose instead the distribution is 
Uniform(0,1), simplifying calculations below.

 For $0\leq t,z\leq 1$, set 
$$\phi_{(\Tree,o)}(z,t):=1-\Ex\left[z^{X_{(\Tree,o)}(t)}\right].$$
For a neighbor $i$ of $o$ (written $i \sim o$), we let $\Tree_i$ denote the subtree of $\Tree$ 
(as viewed from root $o$)  consisting of $i$ and all its descendants. 
\begin{lemma}
\label{L:rec}
For $0\leq z,t\leq 1$,
\begin{eqnarray*}
\label{eq:rec}
\phi_{(\Tree,o)}(z,t)& = & \int_{z}^1 \prod_{i\sim o}\left(1-\int_{0}^t\phi_{(\Tree_i,i)}(\xi,u)du\right)d\xi. 
\end{eqnarray*}
\end{lemma}
\begin{proof}
Let $Y(t)$ denote the fortune of agent $o$ at time $t$ in the modified process where $o$ systematically wins every game she plays. Clearly,
\begin{equation}
\label{eq:decomp}
Y(t)\stackrel{d}{=}1+\sum_{i\sim o}X_{(\Tree_i,i)}(\tau_{oi}){\bf 1}_{\{\tau_{oi}\leq t\}},
\end{equation}
with the terms in the sum being independent.   
The original process can be coupled with the modified process in the natural  way, such that they coincide as long as $o$ has not lost a game (in the original process).
Hence, almost surely under this coupling, 
$$X_{(\Tree,o)}(t) = Y(t) \textrm{ or } X_{(\Tree,o)}(t) = 0.$$ 
From the ``fair game" structure of the CG process we know 
$\Ex\left[X_{(\Tree,o)}(t)\right]=1$.  
But conditioning on the times of meetings does not alter the ``fair game" structure; and because the times of meetings determine $Y(t)$, we have 
$\Ex\left[X_{(\Tree,o)}(t) \ | \ Y(t)\right]=1$.
So the conditional distribution of $X_{(\Tree,o)}(t)$ given $Y(t)$ must be
$$
X_{(\Tree,o)}(t)=
\left\{
\begin{array}{lcl}
Y(t) & \textrm{ w.p } & {1}/{Y(t)}\\
0 & \textrm{ w.p } & 1-{1}/{Y(t)}. 
\end{array}
\right.
.
$$
Hence,
\begin{equation}
\phi_{(\Tree,o)}(z,t) = 
\Ex\  [\Ex[ 1 - z^{X_{(\Tree,o)}(t)} | Y] \ ] 
=
\Ex\left[\frac{1-z^{Y(t)}}{Y(t)}\right]\\
= \int_{z}^1\Ex\left[\xi^{Y(t)-1}\right]d\xi.
\label{qwv}
\end{equation}
Now (\ref{eq:decomp}) gives 
\[ \Ex\left[\xi^{Y(t)-1}\right] = \prod_{i\sim o}\left(1-\int_{0}^t\phi_{(\Tree_i,i)}(\xi,u)du\right)\] 
and substituting into (\ref{qwv}) 
completes the proof. 
\end{proof}

\subsection{Infinite trees}
In the specific contexts below it is straightforward to 
 make rigorous our use of the simple CG process with an infinite number of agents and of the
local weak convergence results (\ref{LWC1},\ref{LWC2}) relating finite and infinite graphs, and we do not give details.
Instead we focus on trying to do explicit calculations.

\paragraph{The infinite $d$-ary tree.}
We next consider the case where $(\Tree,o)$ is the infinite $d-$ary tree 
rooted at $o$,  that is each vertex has $d \ge 1$ children.  
Writing 
 \begin{equation}
 \phi_d(z,t)=1-\Ex\left[z^{X_{(T,o)}(t)}\right]
\label{Tdphi}
\end{equation}
Lemma \ref{L:rec} gives the functional identity
\begin{eqnarray}
\label{eq:recd}
\phi_{d}(z,t)& = & \int_{z}^1 \left(1-\int_{0}^t\phi_d(\xi,u)du\right)^d d\xi. 
\end{eqnarray}
We could not determine $\phi_{d}(z,t)$ explicitly, but we will establish the following bounds, which are sufficient to obtain asymptotics as $d\to\infty$. 
\begin{lemma}
Setting $\varepsilon_d=\frac{2}{d}\log{\left(1+\frac{d}{2}\right)} < 1$, we have for $0\leq z,t\leq 1$ and $d\geq 1$,
$$\frac{2(1-z)(1-\varepsilon_d)}{2(1-\varepsilon_d)+d(1-z)t}\leq \phi_{d}(z,t) \leq \frac{2(1-z)}{2+d(1-z)t}.$$
In particular, $\Pr(X_{(\Tree,o)}(1)  \ne 0) = \phi_{d}(0,1)\sim{2}/{d}$ as $d\to\infty$. 
\end{lemma}
\begin{proof}
First note that, by (\ref{eq:recd}),
\begin{eqnarray*}
-\frac{\partial\phi_{d}}{\partial t}(z,t)& = & d\int_{z}^1 \phi_d(\xi,t)\left(1-\int_{0}^t\phi_d(\xi,u)du\right)^{d-1} d\xi,
\end{eqnarray*}
whereas, using again (\ref{eq:recd}) and the fact that $\phi_d(1,t)\equiv 0$,
\begin{eqnarray*}
\nonumber
\left(\phi_d(z,t)\right)^2 & = & -2\int_{z}^1 \phi_d(\xi,t) \frac{\partial\phi_{d}}{\partial z}(\xi,t)d\xi\\
& = & 2\int_{z}^1 \phi_d(\xi,t)\left(1-\int_{0}^t\phi_d(\xi,u)du\right)^{d} d\xi.
\end{eqnarray*}
Combining these two identities, we obtain
\begin{eqnarray}
\label{eq:identity}
\frac{\partial(1/\phi_d)}{\partial t}(z,t)& = & \frac{d}{2} \ \times \ \frac{\int_{z}^1 \phi_d(\xi,t)\left(1-\int_{0}^t\phi_d(\xi,u)du\right)^{d-1} d\xi}
{\int_{z}^1 \phi_d(\xi,t)\left(1-\int_{0}^t\phi_d(\xi,u)du\right)^{d} d\xi}.
\end{eqnarray}
Since $\phi_d$ is $[0,1]-$valued, we see that
\begin{eqnarray*}
\frac{\partial(1/\phi_d)}{\partial t}(z,t)& \geq & \frac{d}{2} .
\end{eqnarray*}
Integrating with respect to $t$ gives 
\[ \frac{1}{ \phi_d(z,t)} \ge \frac{1}{ \phi_d(z,0)} + \frac{td}{2} 
= \frac{1}{1-z} + \frac{td}{2} \]
which rearranges to the claimed upper bound. 
From this upper bound, it follows that for $0\leq z,t \leq 1$,
\begin{eqnarray}
\label{eq:epsi}
\int_{0}^t\phi_d(z,u)du & \leq & \varepsilon_d, 
\end{eqnarray}
which we may plug into (\ref{eq:identity}) to obtain:
\begin{eqnarray*}
\frac{\partial(1/\phi_d)}{\partial t}(z,t)& \leq & \frac{d}{2} 
\times \frac{1}{1-\varepsilon_d}.
\end{eqnarray*}
Integrating with respect to $t$ gives 
\[ \frac{1}{ \phi_d(z,t)} \le \frac{1}{ \phi_d(z,0)} + \frac{td}{2(1- \varepsilon_d)} 
= \frac{1}{1-z} + \frac{td}{2(1- \varepsilon_d)}  \]
which rearranges to the claimed lower bound. 
\end{proof}

\paragraph{The $r-$regular tree.}
The $r-$regular infinite tree consists of $r$ copies of a $(r-1)-$ary tree, connected to a root. Letting $\phi^*_r(z,t)$ denote the corresponding function, as at (\ref{Tdphi}),  the general recursion from Lemma \ref{L:rec} gives 
\begin{eqnarray*}
\phi^*_{r}(z,t)& = & \int_{z}^1 \left(1-\int_{0}^t\phi_{r-1}(\xi,u)du\right)^r d\xi. 
\end{eqnarray*} 
Comparing with (\ref{eq:recd}) and using (\ref{eq:epsi}), it follows that 
$$(1-\varepsilon_{r-1}) \phi_{r-1}\leq \phi^*_{r}\leq \phi_{r-1},$$
so that $\phi^*_r$ satisfies the same $r\to\infty$ asymptotics as does $\phi_{r-1}$. In particular, 
\begin{corollary}
\label{C:r-tree}
On the $r-$regular infinite tree $\Tree_r$, the probability $\rho(\Tree_r)
 = \phi^*_{r}(0,1)$ that a given agent finishes with non-zero money is $2/r+o(1/r)$ as $r\to\infty$.
\end{corollary}

\paragraph{Galton-Watson trees.} 
When the rooted tree $(T,o)$ is a random Galton-Watson tree with degree distribution $\{\pi_n:n\ge 0\}$,
the general recursion from Lemma \ref{L:rec}
 immediately leads to a recursive distributional equation for the annealed generating function 
$$\phi(z,t):=1-\Ex\left[z^{X_{(T,o)}(t)}\right],$$
where expectation is now taken with respect to both the randomness of the rooted tree and the randomness of the CG process. Letting $F_\pi(x)=\sum_{n}\pi_n x^n$ denote the degree generating function of the Galton-Watson tree, we readily obtain:
\begin{eqnarray}
\phi(z,t)& = & \int_{z}^1 F_\pi\left(1-\int_{0}^t\phi(\xi,u)du\right) d\xi. 
\label{eq:GW}
\end{eqnarray}
Extracting useful information from this equation for a general distribution 
$\{\pi_n:n\ge 0\}$ remains an open problem.

\subsection{The Poisson-Galton-Watson tree, the sparse \ER graph and the short-time behavior of the Kingman coalescent}
\label{sec:ERKC}

In the case where $\{\pi_n:n\ge 0\}$ is the Poisson distribution with mean $c\geq 0$ (i.e. $F_\pi(z)=e^{cx-c}$),  equation (\ref{eq:GW}) can be easily solved, yielding the following explicit formula:
\begin{eqnarray*}
\phi(z,t)& = & \frac{2(1-z)}{2+c(1-z)t}.
\end{eqnarray*}
Identifying this generating function, we find that the fortune $X(t)$ of the agent at the root has distribution specified by 
\begin{eqnarray}
&\Pr (X(t) > 0) = \sfrac{2}{2+ct} & \label{PGW1}\\
 \mbox{the conditional distribution of } & \mbox{$X(t)$ given $X(t) > 0$ } 
& \mbox{
is Geometric($\frac{2}{2+ct}$).} \label{PGW2}
\end{eqnarray}
Because the Poisson-Galton-Watson tree $\Tree^*_c$ is the local weak limit
of the sparse \ER random graph $\GG(n,c/n)$
we can deduce
\begin{equation}
 \Ex \rho(\GG(n,c/n)) \to \Ex \rho(\Tree^*_c) = \Pr(X(1) > 0) = \sfrac{2}{2+c} . \label{LWC2}
\end{equation}
Let us outline an interesting alternative explanation of why 
(\ref{PGW1}, \ref{PGW2}) arise here. 
Under our time-change (first meetings occur at Uniform$(0,1)$ random times)
the simple CG process on $\GG(n,c/n)$ arises from a two-stage construction: for each edge $e$ of the complete graph on $n$ vertices, first select $e$ with probability $c/n$, then (if selected) assign the 
Uniform random meeting time.  
But the set of meetings that occur before time $t \in [0,1]$ can alternatively 
be described by: 
for each edge $e$ of the complete graph, a meeting has occured with 
chance $ct/n$, independently over $e$, and meetings occured at 
independent Uniform times. 

Now consider the Kingman coalescent, 
in the spirit of more general stochastic coalescence 
models
\cite{me78,MR2574323,MR2253162,EC}, 
as a process of coalescing partitions 
of $\{1,2,\ldots,n\}$, being the special case in which each pair of blocks merges at constant rate.  
But take this rate to be $1/n$ instead of $1$.
The $n \to \infty$ limit distribution of block sizes at time $\tau$  
in this short-time limit regime is known to be Geometric($\frac{2}{2+\tau}$) 
-- see  Construction 5 in \cite{me78} for an intuitive explanation in terms of a process of coalescing intervals on $\Ints$.
But in our ``CG process on $\GG(n,c/n)$" above,  the process of fortunes of the solvent agents, considered as a process indexed by $\tau = ct$, 
evolves in essentially the same way as the process of block sizes in 
this Kingman coalescent, 
so (\ref{PGW2}) is ultimately equivalent to the short-time Geometric limit 
result for the Kingman coalescent.

\newpage
%\bibliographystyle{amsplain}
%\bibliography{CG}

\providecommand{\bysame}{\leavevmode\hbox to3em{\hrulefill}\thinspace}
\providecommand{\MR}{\relax\ifhmode\unskip\space\fi MR }
\providecommand{\MRhref}[2]{%
  \href{http://www.ams.org/mathscinet-getitem?mr=#1}{#2}
}
\providecommand{\href}[2]{#2}

\end{document}